\documentclass[times,twocolumn,final,authoryear]{elsarticle}

\usepackage{prletters}
\usepackage{framed,multirow}

\usepackage{amssymb}
\usepackage{latexsym}

\usepackage{url}
\usepackage{xcolor}
\definecolor{newcolor}{rgb}{.8,.349,.1}

\usepackage{amsthm}
\usepackage{tikz}

\def\Fc{{\mathcal F}}
\def\Sc{{\mathcal S}}
\def\H{{\mathcal H}}

\def\kb{{\mathbf k}}
\def\1{{\mathbf 1}}
\def\xb{{\mathbf x}}

\def\lcm{{\rm lcm}}
\def\prob{{\rm prob}}

\newtheorem{Theorem}{Theorem}[section]
\newtheorem{Proposition}[Theorem]{Proposition}

\newtheorem{Corollary}[Theorem]{Corollary}
\newtheorem{Remark}[Theorem]{Remark}
\newtheorem{Example}[Theorem]{Example}

\begin{document}

\begin{frontmatter}

\title{Efficient multicut enumeration of $k$-out-of-$n$:F and consecutive $k$-out-of-$n$:F systems}

\author[1]{Fatemeh {Mohammadi}}
\author[2]{Eduardo {S\'aenz-de-Cabez\'on}\corref{cor1}} 
\ead{eduardo.saenz-de-cabezon@unirioja.es}
\author[3]{Henry P.  {Wynn}}

\address[1]{School of Mathematics, University of Bristol, Bristol, BS8 1TW, UK}
\address[2]{Universidad de La Rioja, Departamento de Matem\'aticas y Computaci\'on, Logro\~no 26005, Spain}
\address[3]{Department of Statistics, London School of Economics, London, WC1X 9BA, UK }

\begin{abstract}
We study multiple simultaneous cut events for  $k$-out-of-$n$:F and linear consecutive $k$-out-of-$n$:F systems in which each component has a constant failure probability. We list the multicuts of these systems and describe the structural differences between them. Our approach, based on combinatorial commutative algebra, allows complete enumeration of  the sets of multicuts for both kinds of systems. We also analyze the computational issues of multicut enumeration and reliability computations. 
\end{abstract}


\end{frontmatter}


A system $S$ is formed by several components each of which can perform at different levels with different probabilities. The performance of the system is the result of the combination of the performance states of its components. A system is called {\em coherent} if the degradation of the performance of any of its components cannot lead to an improvement of the system's performance and the improvement of any component cannot lead to degradation of the system's performance. Coherent systems include most industrial and biological systems, networks or workflow diagrams. The reliability of a coherent system $S$ is the probability that the system is performing at a certain working level. The unreliability of $S$ is defined as the probability that it is performing below the required level. Here we focus on system's unreliability, but our methods can directly be applied to its reliability. The study of system (un)reliability is a central area in engineering. In the recent years new mathematical methods and approaches have been developed including combinatorial methods, the study of correlated probabilities or minimal path and cut enumeration, cf. the recent papers \cite{LZXLW16, GN17, SZMZG16, WWXC17, LB17, BAGT16}.

Let $S$ be a coherent system with $n$ components, each of which has a constant probability of failure $p_i,\, i\in\{1,\dots,n\}$. If all $p_i$'s are equal and independent then the system is said to have i.i.d. (independent and identically distributed) components. We will consider that the failure probabilities of the components are mutually independent but do not need to be identical. The method does however apply to more general situations in which independency is not assumed. For the system $S$ we define a set of {\em minimal} failure states, i.e. states under which the system is failed and such that the improvement of any of its components would lead the system to a working state. We call such minimal failure states {\em minimal cuts}, while any failure state is simply called a {\em cut}. The unreliability of $S$ is described as the probability that at least one minimal failure event occurred.

We are interested in the behaviour of  the system $S$ under several simultaneous failure events. By studying the distribution of the variable $Y=\{\mbox{number of minimal failure events}\}$ we obtain further information about the structure and behaviour of $S$. This is relevant for the management of spare parts, supplies and maintenance of the system. The consideration of simultaneous failures is important for instance in high performance computational systems, see \cite{TNLP13}. Analogous to cuts we define $i$-multicuts as those system states that correspond to $i$ simultaneous cuts. A minimal $i$-multicut is a state in which the system is suffering $i$ simultaneous cuts but the improvement of any component makes the number of simultaneous cuts strictly smaller than $i$. Consider for instance the {\em series} system $S$ with $3$ components. A minimal cut of $S$ is for instance $\{1\}$ which means that the component $1$ is failing. The cut given by $\{1,2,3\}$ is a non-minimal cut, a non-minimal $2$-multicut and a minimal $3$-multicut for $S$.

\begin{center}
\begin{tikzpicture}[scale=0.6, transform shape]
\fill (0,0) circle(3pt);
\fill (2,0) circle(3pt);
\fill (4,0) circle(3pt);
\fill (6,0) circle(3pt);
\draw (0,0)--(2,0) node[midway,above]{1};
\draw (2,0)--(4,0)node[midway,above]{2};
\draw (4,0)--(6,0)node[midway,above]{3};
\end{tikzpicture}
\end{center}
 
The probability that a system $S$ is suffering at least $i$ simultaneous failures is the union event of all minimal $i$-multicuts of $S$. A  first step to compute this probability is to enumerate the combinations of $i$ simultaneous failures for each $i\geq 1$ and then compute the corresponding probabilities. This enumeration is hard in general for the number of minimal $i$-multicuts can grow exponentially.

$k$-out-of-$n$:F and consecutive $k$-out-of-$n$:F systems (Cons/$k$-out-of-$n$:F) are fundamental examples of coherent systems that generalize the basic series and parallel systems. These systems are ubiquitously studied and applied because of their redundancy, which is useful in a large variety of situations, see \cite{ZLW00, HZF03, E10, GGK16}. The tool we use to analyze multicuts and unreliability of those systems is the algebraic approach described by the authors in \cite{SW09,SW10,SW11,SW15}. This approach consists in associating a monomial ideal $I_S$ to the system $S$ and analyze the features of $S$ by studying the properties of $I_S$.
The first section of the paper describes the algebraic approach to system reliability and its application to the problem of simultaneous failure events.  In \S\ref{sec:kn} and \S\ref{sec:consecutive} we focus on $k$-out-of-$n$:F systems and Cons/$k$-out-of-$n$:F systems, and give a description of their $i$-multicut structures, demonstrating their significant differences. Finally in \S\ref{sec:computations} we consider some of the computational aspects of the problem and show the results of computer experiments based on the implementation of the tools previously described.

\section{Algebraic reliability primer}\label{sec:introduction}
Let $S$ be a coherent system with $n$ components. For ease of notation we assume that $S$ is a binary system, i.e. the system and each of its components $c_1,\dots,c_n$ can be either failed $(1)$ or working $(0)$. However, this method works also for multistate components. Let $\Fc_S$ be the set of failure states (cuts) of $S$ and let $\overline{\Fc}_S$ be the set of minimal cuts. Let $R=\kb[x_1,\dots,x_n]$ be a polynomial ring over a field $\kb$. Each state of $S$ corresponds to a monomial in $R$: the monomial $x_{i_1}\cdots x_{i_r}$ for $1\leq i_1<\cdots<i_r\leq n$ corresponds to the state in which components $c_{i_1},\dots , c_{i_r}$ are failed and the rest are working. The coherence property of $S$ is equivalent to the fact that the elements of $\Fc_S$ correspond to the monomials in a monomial ideal, the  {\em failure ideal} of $S$, which we denote by $I_S$. The unique minimal monomial generating set of $I_S$ is formed by the monomials corresponding to the elements of $\overline{\Fc}_S$, see \cite[\S 2]{SW09}. Obtaining the set of minimal cuts of $S$ amounts to compute the minimal generating set of $I_S$.
 In order to compute the unreliability of $S$ we can use the numerator of the Hilbert series of $I_S$ which gives a formula, in terms of $x_1,\dots,x_n$ that enumerates all the monomials in $I_S$, i.e. the monomials corresponding to states in $\Fc_S$. Hence, computing the (numerator of the) Hilbert series of $I_S$ provides a way to compute the unreliability of $S$ by substituting $x_i$ by $p_i$, the failure probability of the component $i$ as explored in \cite[\S 2]{SW09}. Furthermore, in order to have a formula that can be truncated to obtain bounds for the reliability in the same way that the inclusion-exclusion formula is truncated to obtain the so-called Bonferroni bounds, we need a special way to write the numerator of the Hilbert series of $I_S$. This convenient form is given by the alternating sum of the ranks in any free resolution of $I_S$. Every monomial ideal has a {\em minimal} free resolution, which provides the tightest bounds among the aforementioned ones. In general, the closer the resolution is to the minimal one, the tighter the bounds obtained, see e.g. section $3$ in \cite{SW09}.

In summary, the algebraic method for computing the unreliability of a coherent system $S$ works as follows: 
\begin{itemize}
\item[-]{Associate to the system $S$ its failure ideal $I_S$ and obtain the minimal generating set of $I_S$ to get the set $\overline{F}_S$ of minimal cuts of $S$.}
\item[-]{Compute the Hilbert series of $I_S$ to have the unreliability of $S$.}
\item[-]{Compute any free resolution of $I_S$. The alternating sum of the ranks of this resolution gives a formula for the Hilbert series of $I_S$, i.e. the unreliability of $S$ which provides bounds by truncation at each summand. If the computed resolution is minimal, then these bounds are the tightest of this type.}
\end{itemize}

Observe that this method works without alterations if the failure probabilities of the components of $S$ are mutually different. For more details and the proofs of the results enunciated here, see \cite{SW09}. For more applications of this method in reliability analysis, see \cite{SW10,SW11,SW15}.

When studying multiple simultaneous failures in a coherent system, the above methods can be extended to a filtration of ideals associated to the system under study as follows (cf. \cite{MSW17}): Let $S$ be a coherent system in which several minimal failures can occur at the same time. Let $Y$ be the number of such simultaneous failures. The event $\{Y\geq 1\}$ is the event that at least one elementary failure event occurs, which is the same as the event that the system fails. If $x^\alpha$ and $x^\beta$ are the monomials corresponding to two elementary failure events, then $\lcm(x^\alpha,x^\beta)=x^{\alpha\vee\beta}$ corresponds to the intersection of the two events. The full event $\{Y\geq 2\}$ corresponds to the ideal generated by all such pairs. The argument extends to $\{Y\geq i\}$  and the study of the tail probabilities $\prob\{Y\geq i\}$. 
For each $i$, the set of $i$-multicuts generates an ideal, which altogether form a filtration. Let $I\subseteq \kb [x_1,\dots,x_n]$ be a monomial ideal and $\{m_1,\dots,m_r\}$ its minimal monomial generating set. Let $I_i$ be the ideal generated by the least common multiples of all sets of $i$ distinct monomial generators of $I$, i.e. 
$I_i=\langle \lcm(\{m_j\}_{j\in\sigma}) : \sigma\subseteq\{1,\dots,r\}, \vert\sigma\vert=i\rangle$.
We call $I_i$ the \emph {$i$-fold $\lcm$-ideal of $I$}. The ideals $I_i$ form a descending filtration $I=I_1\supseteq I_2\supseteq \cdots \supseteq I_r$, which we call the \emph{$\lcm$-filtration} of $I$. If $I=I_S$ is the failure ideal of a system $S$ then the minimal generating set of ${I}_i$ is formed by the monomials corresponding to minimal $i$-multicuts of $S$.
The survivor functions $F(i)=\prob\{Y\geq i\}$ for a coherent system, are obtained from the multigraded Hilbert function of the $i$-fold $\lcm$-ideal ${I}_i$. The above considerations can be summarized in the following proposition.

\begin{Proposition}
Let $Y$ be the number of failure events of a system $S$. If $\{ m_1,\dots, m_r\}$ is the set of monomial minimal generators of the failure ideal $I_S=I_1$
then $\prob\{Y\geq i\}=\mbox{E} [\1_{M_i}(\alpha)]=\H_{I_i}(\xb)$,
where $\1_{M_i}$ is the indicator function of the exponents of monomials in the $i$-fold $\lcm$-ideal
$I_i$
and $\H_{I_i}(\xb)$ is the numerator of its multigraded Hilbert series.
\end{Proposition}
As a result, we also obtain identities and lower and upper bounds for multiple failure probabilities from free resolutions as in the single failure case.

\section{Multiple simultaneous failures in $k$-out-of-$n$:F systems}\label{sec:kn}
A $k$-out-of-$n$:F system denoted by $S_{k,n}$ is a system with $n$ components that fails whenever $k$ of them fail. The failure ideal of $S_{k,n}$ is given by $I_{k,n}=\langle \prod_{i\in\sigma} x_i:\  \sigma \subseteq \{1,\dots,n\}, \vert\sigma\vert=k\rangle$. Let $I_{k,n,i}$ be the $i$-fold $\lcm$-ideal of $I_{k,n}$. 

\begin{Theorem}\label{thm:kn}
For all $k<j\leq n$ and ${j-1\choose k}<l\leq{j \choose k}$ we have that $I_{k,n,l}=I_{j,n}=\langle \prod_{s\in\sigma} x_s:\ \sigma \subseteq \{1,\dots,n\}, \vert\sigma\vert=j\rangle$
\end{Theorem}

\begin{proof}
Let $k<j\leq n$. We want to see that $I_{j,n}=I_{k,n,l}$ for all $l$ such that ${{j-1}\choose k}< l\leq{j\choose k}$. This amounts to show that for all such $l$ the following apply:
\begin{enumerate}
\item[(i)]{Any subset of $\{1,\dots,n\}$ of cardinality $j$ (i.e. every minimal generator of $I_{j,n}$) can be obtained as the union of $l$ subsets of $\{1,\dots,n\}$ of cardinality $k$.}
\item[(ii)]{No such subset of cardinality $j$ can be obtained as the union of more than ${j\choose k}$ subsets of cardinality $k$.}
\item[(iii)]{Taking any union of less than ${{j-1}\choose k}+1$ subsets of cardinality $k$ produces some sets of cardinality less than $j$.}
\end{enumerate}

To see (i) observe that to obtain a set of cardinality $j$ we only need the union of $\lceil \frac{j}{k}\rceil$ sets of cardinality $k$. Just consider $\lfloor\frac{j}{k}\rfloor$ sets of $k$ disjoint elements, which is always possible, and add another one which contains the remaining $j-\lfloor\frac{j}{k}\rfloor$ elements.
To see (ii) observe that any union of more than ${j\choose k}$ sets of cardinality $k$ has at least cardinality $j+1$. This is because for every set of $j$ elements, we can choose $k$ of them in at most ${j\choose k}$ ways. Hence if we have more than ${j\choose k}$ sets of cardinality $k$, their union must have more than $j$ elements.
To see (iii) consider any set of $j-1$ elements. These can be chosen in sets of cardinality $k$ in ${{j-1}\choose k}$ ways. The union of all these ${{j-1}\choose k}$ sets has then $j-1$ elements.

Using (i), (ii) and (iii) we have that for all ${{j-1}\choose k}+1\leq l\leq{j\choose k}$ $I_{k,n,l}$ is generated by some subsets of $\{1,\dots,n\}$ all of which have cardinality $j$. We still have to check that it is generated by all such sets. For this, take any subset $\sigma\subseteq\{1,\dots,n\}$ such that $\vert\sigma\vert=j$. Let $i$ be an element of $\sigma$ and $\sigma'=\sigma - {i}$. We can form $\sigma$ as the union of $l$ sets of cardinality $k$ for all ${{j-1}\choose k}+1\leq l\leq{j\choose k}$. Observe that $\sigma'$ can be obtained as the union of ${{j-1}\choose k}$ sets of cardinality $k$. Now, we can add to this union any number of sets of cardinality $k$ that contribute only with the element $i$ to it. We can do this with any number of sets from $1$ to ${{j-1}\choose{k-1}}$. Since ${{j-1}\choose k}+{{j-1}\choose{k-1}}={j\choose k}$ the proof is complete.
%
\end{proof}

Theorem \ref{thm:kn} shows that the minimal multicuts of $k$-out-of-$n$:F systems are ordinary minimal cuts for other $j$-out-of-$n$:F systems. On the other hand, the simultaneous failure ideals follow a {\em staircase} pattern, in the sense that their reliability is the same for several numbers of simultaneous failures until a certain particular number occurs.
Theorem~\ref{thm:kn} also enables us to study the reliability polynomial applying geometric techniques developed in \cite{mohammadi1}.
The fact that the ideals in the lcm-filtration of $I_{k,n}$ correspond to the failure ideals of other $j$-out-of-$n$ systems allows us to obtain reliability identities and bounds by making use of the known formulae for the Betti numbers of these systems, cf. \cite{SW09}.
\begin{Corollary}
In the case of i.i.d. components, the graded Betti numbers of $I_{k,n,i}$ are given by 
\[
\beta_{a,a+j}(I_{k,n,i})=\beta_{a,a+j}(I_{j,n})={{n} \choose {j+a}} {{a+j-1} \choose {j-1}}
\]
for all ${j-1\choose k}<i\leq{j \choose k}$ and $0\leq a\leq n-j$.
For $b\neq a+j$ we have $\beta_{a,b}(I_{k,n,i})=0$.
For the multigraded Betti numbers, we have that
\[\beta_{a,\alpha}(I_{k,n,i})=\beta_{a,\alpha}(I_{j,n})={{a+j-1} \choose {j-1}}
\]
if $\alpha$ is the product of $a+j$ variables and $0$ otherwise.
\end{Corollary}

\begin{Remark}
From the graded Betti numbers we obtain the unreliability polynomial of the $j$-out-of-$n$ systems as 
\begin{eqnarray*}
f_{j,n}(x)=&\sum_{a=0}^{n-j}(-1)^a\beta_{a,a+j}(I_{j,n})x^{a+j}\\
=&\sum_{a=0}^{n-j}(-1)^a{{n} \choose {j+a}} {{a+j-1} \choose {j-1}}x^{a+j}.
\end{eqnarray*}
By truncating this polynomial we obtain bounds for the reliability formula.
The multigraded version is the following
\[
f_{j,n}(x)=\sum_{a=0}^{n-j}(-1)^a{{a+j-1} \choose {j-1}} \sum_{\alpha\in[n,j+a]}x^{\alpha}
\]
where $[n,j+a]$ denotes the set of vectors with $1$ in the indices of the $(j+a)$-subsets of $\{1,\dots,n\}$ and 0 in the other entries.
\end{Remark}

\begin{Example}
{\rm Let us consider a $k$-out-of-$n$:F system for $k=2$ and $n=8$. The  failure ideal $I_{2,8}$ of this system is generated by all the $28$ possible pairs of $8$ variables. The $\lcm$-ideals that describe the multiple failures of this system are given in Table~\ref{table:lcm}.

\begin{table}[h]
\footnotesize
\begin{tabular}{r|l}
$I_{2,8}$&$I_{2,8,1}$\cr
$I_{3,8}$&$I_{2,8,2}=I_{2,8,3}$\cr
$I_{4,8}$&$I_{2,8,4}=I_{2,8,5}=I_{2,8,6}$\cr
$I_{5,8}$&$I_{2,8,7}=I_{2,8,8}=I_{2,8,9}=I_{2,8,10}$\cr
$I_{6,8}$&$I_{2,8,11}=I_{2,8,12}=I_{2,8,13}=I_{2,8,14}=I_{2,8,15}$\cr
$I_{7,8}$&$I_{2,8,16}=I_{2,8,17}=I_{2,8,18}=I_{2,8,19}=I_{2,8,20}=I_{2,8,21}$\cr
$I_{8,8}$&$I_{2,8,22}=I_{2,8,23}=I_{2,8,24}=I_{2,8,25}=I_{2,8,26}=I_{2,8,27}=I_{2,8,28}$\cr
\end{tabular}
\medskip
\caption{lcm- ideals (left) and multicut ideals (right) of the failure ideals of $S_{2,8}$}
\label{table:lcm}
\end{table}

Figure \ref{fig:unreliabilitykn} shows the unreliability of i.i.d. $j$-out-of-$8$:F systems for $j$ from $2$ to $8$, each of which corresponds to the probability of multiple simultaneous failures of the original system. For example, the unreliability of $5$-out-of-$8$ system corresponds to $7,\,8,\,9$ or $10$ simultaneous failure events in the $2$-out-of-$8$ system.

Figure \ref{fig:unreliabilityfiltration} shows the probability that multiple failures appear in $S_{2,8}$ for three different component failure probabilities $p=0.2,\,0.5$ and $0.8$. This graph shows the {\em staircase} behaviour of $k$-out-of-$n$:F systems subject to several simultaneous failures.
For high failure probability, e.g. $p=0.8$, the probability of multiple failures drops slowly at the beginning but decreases more rapidly as the number of simultaneous failures grows. On the other hand, when the failure probability is lower, the probability of simultaneous failures drops drastically with fewer simultaneous failure events.

\begin{figure}[t]
\begin{center}
\includegraphics[scale=0.35]{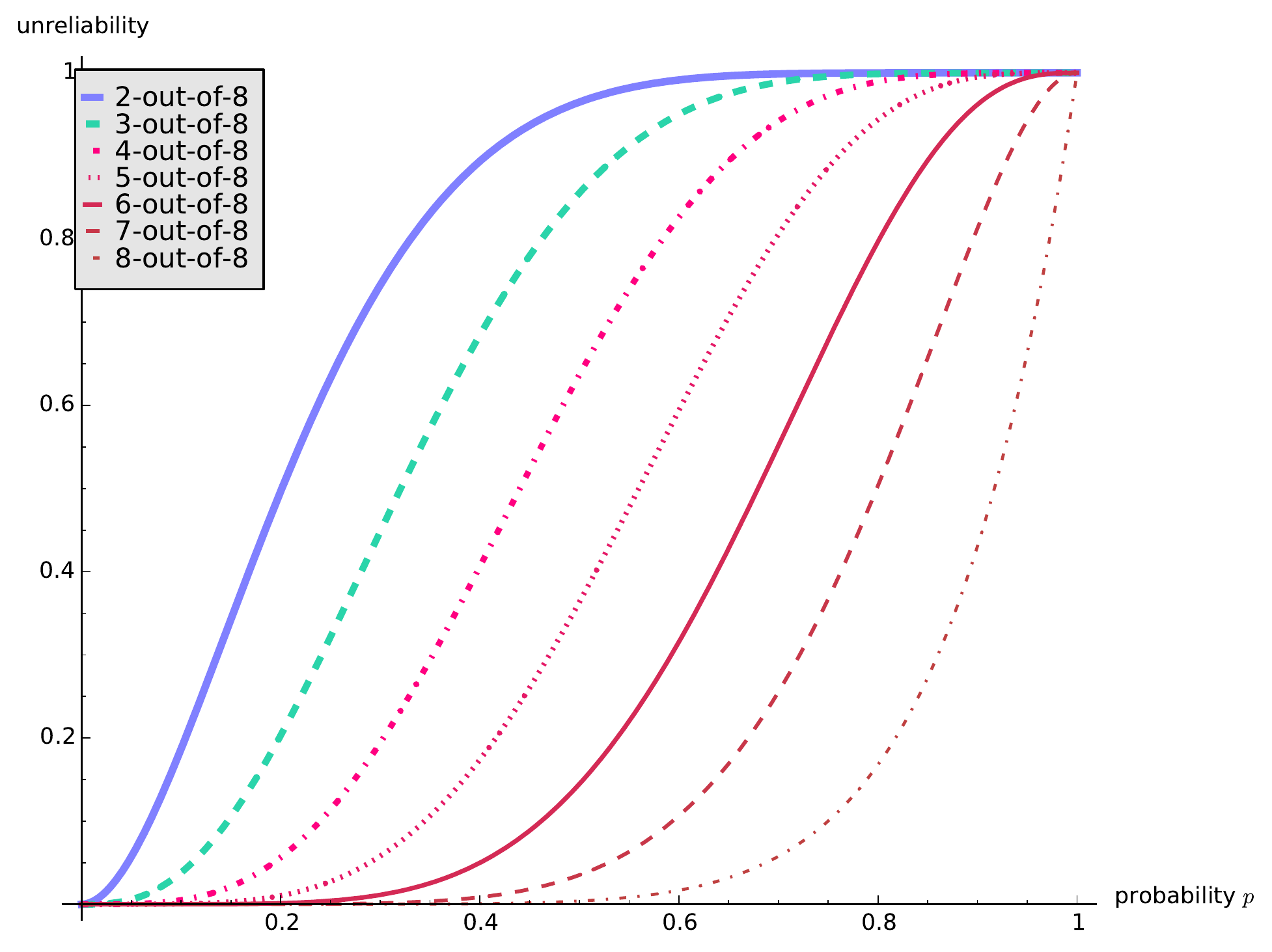}
\caption{Unreliability of i.i.d. $j$-out-of-$8$ systems for $j$ from $2$ to $8$.}\label{fig:unreliabilitykn}
\end{center}
\end{figure}

\begin{figure}[t]
\begin{center}
\includegraphics[scale=0.35]{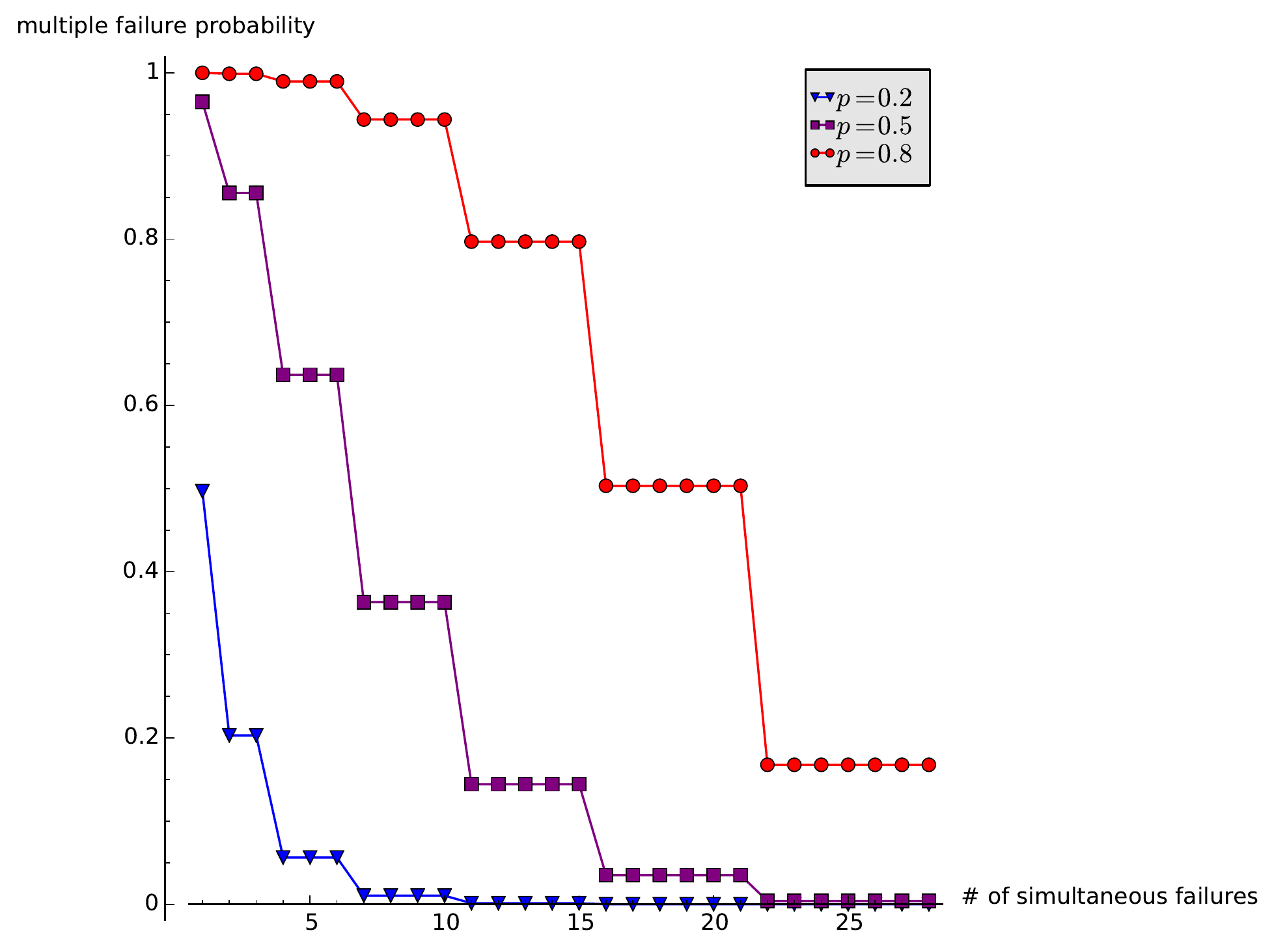}
\caption{Probability of simultaneous failures for the i.i.d. $2$-out-of-$8$ system with different component failure probabilities.}\label{fig:unreliabilityfiltration}
\end{center}
\end{figure}
}
\end{Example}

\section{Cons/$k$-out-of-$n$:F systems}\label{sec:consecutive}

A Cons/$k$-out-of-$n$:F system $T_{k,n}$ has $n$ components and fails whenever $k$ {\em consecutive} components fail. The ideal of $T_{k,n}$ is given by  $J_{k,n}=\langle x_1 \cdots x_k,x_2 \cdots x_{k+1}, \dots,x_{n-k+1} \cdots x_n\rangle$. Let $J_{k,n,i}$ be the $i$-fold $\lcm$-ideal of $J_{k,n}$, generated by the monomials corresponding to $i$ simultaneous minimal failures of $T_{k,n}$.

Let us denote by $\Sc$ the set of subsets of  $\{1,\dots,n-k+1\}$. For any $\sigma$ in $\Sc$ let $\min(\sigma)$ and $\max(\sigma)$ denote the minimum and maximum element of $\sigma$ respectively. Let $\Sc^i$ denote the elements of $\Sc$ of cardinality $i$. We say that $\sigma\in\Sc$ has a gap of size $s$ if there is a subset of $s$ consecutive elements of $\{\min(\sigma),\dots, \max(\sigma)\}$ that are not in $\sigma$. For instance, $\sigma= \{2,3,6,8\}$ has a gap of size $2$ (because $4$ and $5$ are not in $\sigma$) and a gap of size $1$ because $7$ is not in $\sigma$. Let $\Sc_a$ be the set of those $\sigma\in\Sc$ such that the smallest gap in $\sigma$ has size $a$. Let $\Sc^i_a$ be the set of all elements in $\Sc_a$ of cardinality $i$.

\begin{Theorem}\label{thm:conskn}
The ideal $J_{k,n,i}$ is minimally generated by the monomials $m_\sigma$ with $\sigma\in\Sc^i_0\cup\Sc^i_{k}\cup\Sc^i_{k+1}\cup\cdots$. In other words, the minimal generators of $J_{k,n,i}$ correspond to the $\lcm$'s of sets of monomials of cardinality $i$ with no gaps of sizes  $1, 2,\ldots,k-1$.
\end{Theorem}
\begin{proof}
Let $m_\sigma=\lcm(\{m_j\}_{j\in\sigma})$. We know that the set $\{m_\sigma,\,\sigma\in\Sc,\ \vert\sigma\vert=i\}$
is a generating set of $J_{k,n,i}$. Let $G_{k,n,i}$ be the minimal set of generators of $J_{k,n,i}$. To prove the theorem we must show that
\begin{enumerate}
\item[(i)] $\Sc^i_0\subseteq G_{k,n,i}$.
\item[(ii)] If $\sigma\in\Sc^i_a$ with $1\leq a \leq k-1$ then there is another element $\tau$ in $\Sc^i_b$ with $b=0$ or $b>a$ such that $m_\tau$ divides $m_\sigma$ (and hence $m_\sigma\notin G_{k,n,i}$ i.e. it is not minimal). 
\item[(iii)] If $\sigma\in\Sc^i_a$ with $a\geq k$ then $m_\sigma\in G_{k,n,i}$.
\end{enumerate}

To see (i) observe that if $\sigma\in\Sc^i_0$ then $m_\sigma$ is the $\lcm$ of $i$ consecutive minimal generators of $J_{k,n}$, say $m_{a},\dots,m_{a+i}$. Hence $m_\sigma=x_{a}\cdots x_{a+i+k-1}$ which is a product of $i+k-1$ variables. Using $i$ different subsets of $\{1,\dots,n-k+1\}$ we obtain at least $i+k-1$ variables. Note that we can only achieve this minimum if the $i$ elements of $\sigma$ are consecutive, i.e. if we are in $\Sc^i_0$. Hence, if there is another $\sigma'\in\Sc^i$ that divides $\sigma$ then it must be $\sigma$ itself as any other element of $\Sc^i_0$ produces a different set of variables and would not divide $m_\sigma$.

Now, to see (ii) let us assume without loss of generality that $\sigma$ is formed by two blocks $\sigma_1$ and $\sigma_2$ of consecutive elements, separated by a gap of size $a$ with $1\leq a\leq k-1$. For ease of notation we can consider $\sigma_1=\{1,\dots,i_1\}$ and $\sigma_2=\{i_1+a,\dots,i_2\}$. Then $m_\sigma=(x_1\cdots x_{i_1+k-1})\cdot (x_{i_1+a}\cdots x_{i_2+k-1}).$ But now, since $a<k$ we have that $\sigma'={\{1,\dots,i_1+1\}\cup\{i_1+a+1,\dots,i_2\}}$ is in $\Sc^i_b$ with $b\leq a$ (in fact it is either in $\Sc_0$ or in $\Sc_a$) and $m_{\sigma'}$ divides $m_\sigma$. In case $b\neq 0$, we repeat this procedure until we arrive at $b=0$, where $m_{\{ 1,\dots,i\}}$ divides $m_\sigma$ and hence $m_\sigma$ is not a minimal generator of $J_{k,n,i}$. Observe that the set $\{ 1,\dots,i\}$ is in $\Sc_0$. We could have also proceeded in a symmetric way, i.e. using $\sigma'={\{1,\dots,i_1-1\}\cup\{i_1+a-1,\dots,i_2\}}$ in $\Sc^i_b$ and we would obtain another element in $\Sc^i_0$ that divides $\sigma$.

Finally, to see (iii) let us assume again, without loss of generality, that  $\sigma_1=\{1,\dots,i_1\}$ and $\sigma_2=\{i_1+a,\dots,i_2\}$ with $a\geq k$. Now, if $\tau\in\Sc^i$ is different from $\sigma$ then there is at least one index $\tau_j$ such that $\tau_j\neq\sigma$. If $\tau_j>i_2$ then there exists at least a variable bigger than $x_{i_2+k-1}$ in $\tau$ which is not in $\sigma$ and hence $m_\tau$ does not divide $m_\sigma$. The same happens if $i_1<j<i_1+a$: since $a\geq k$ every $m_j$ with $i_1<j<i_1+a$ has a variable that does not belong to $\{x_1,\dots,x_{i_1+k-1}\}\cup \{x_{i_1+a},\dots, x_{i_2+k-1}\}$. To see this observe that $i_1+a>i_i+k-1$ and every $m_j$ with $i_1<j<i_1+a$ has at least one variable in the nonempty set $\{x_{i_1+1},\dots,x_{i_1+a-1}\}$. Hence $\tau$ does not divide $\sigma$ and so $\sigma$ is in $G_{k,n,i}$.
\end{proof}

\begin{Remark}
Each $\tau\subseteq\{1,\dots,n\}$ is given by blocks formed by subsequent indices (i.e. with no gaps among them), for instance, $123679$ is formed by a block of three indices $(123)$ a block of two indices $(67)$ and a block of one index $(9)$. From these blocks we can read the degree of the corresponding generator: if $\tau$ is formed by $l$ blocks of sizes $s_1, s_2, \dots,s_l$ then $m_\tau$ has degree $(s_1+1)+(s_2+1)+\cdots+(s_l+1)$.
\end{Remark}

\begin{Example}\label{ex:294}
{\rm 
$J_{2,9,4}$ is minimally generated by $26$ monomials that correspond to taking $\lcm$'s of the following sets of generators of $J_{2,9}$:
\begin{table}[h]
\footnotesize
\begin{tabular}{c|c|c}
Block sizes&sets&deg. of gens. \\
\hline
4&1234,2345,3456,4567,5678&5\\
\hline
3,1&1236,1237,1238,2347,2348,3458&6\\
 &1456,1567,2567,1678,2678,3678&6\\
\hline
2,2&1256,1267,1278,2367,2378,3478&6\\
\hline
2,1,1&1258,1458,1478&7\\
\end{tabular}
\end{table}

Observe  that $1236$ in the table means $\lcm$ of generators $1,2,3$ and $6$ (i.e. $x_1x_2,x_2x_3,x_3x_4$ and $x_6x_7$), so $1236$ corresponds to $x_1x_2x_3x_4x_6x_7$, which means that the failure of the components $1,2,3,4,6$ and $7$ is a minimal $4$-multicut for $T_{2,9}$.
}
\end{Example}

\begin{Example}
{\rm 
Let us consider a Cons/$2$-out-of-$9$:F system. Its failure ideal is $J_{2,9}=\langle x_1x_2,x_2x_3,x_3x_4,x_4x_5,x_5x_6,x_6x_7,x_8x_9\rangle.$
Figure \ref{fig:consUnrel} shows on one hand the unreliability polynomials of Cons/$i$-out-of-$9$:F systems for $i$ between $2$ and $9$, and on the other hand the polynomials corresponding to $i$ multiple simultaneous failures of a Cons/$2$-out-of-$9$:F system for $i$ from $1$ to $8$. Observe that the multiple failures are, unlike in the ordinary $k$-out-of-$n$:F case, different from other Cons/$k$-out-of-$n$:F systems; the graphs reflect this different behaviour.

\begin{figure}[t]
\begin{center}
\includegraphics[scale=0.35]{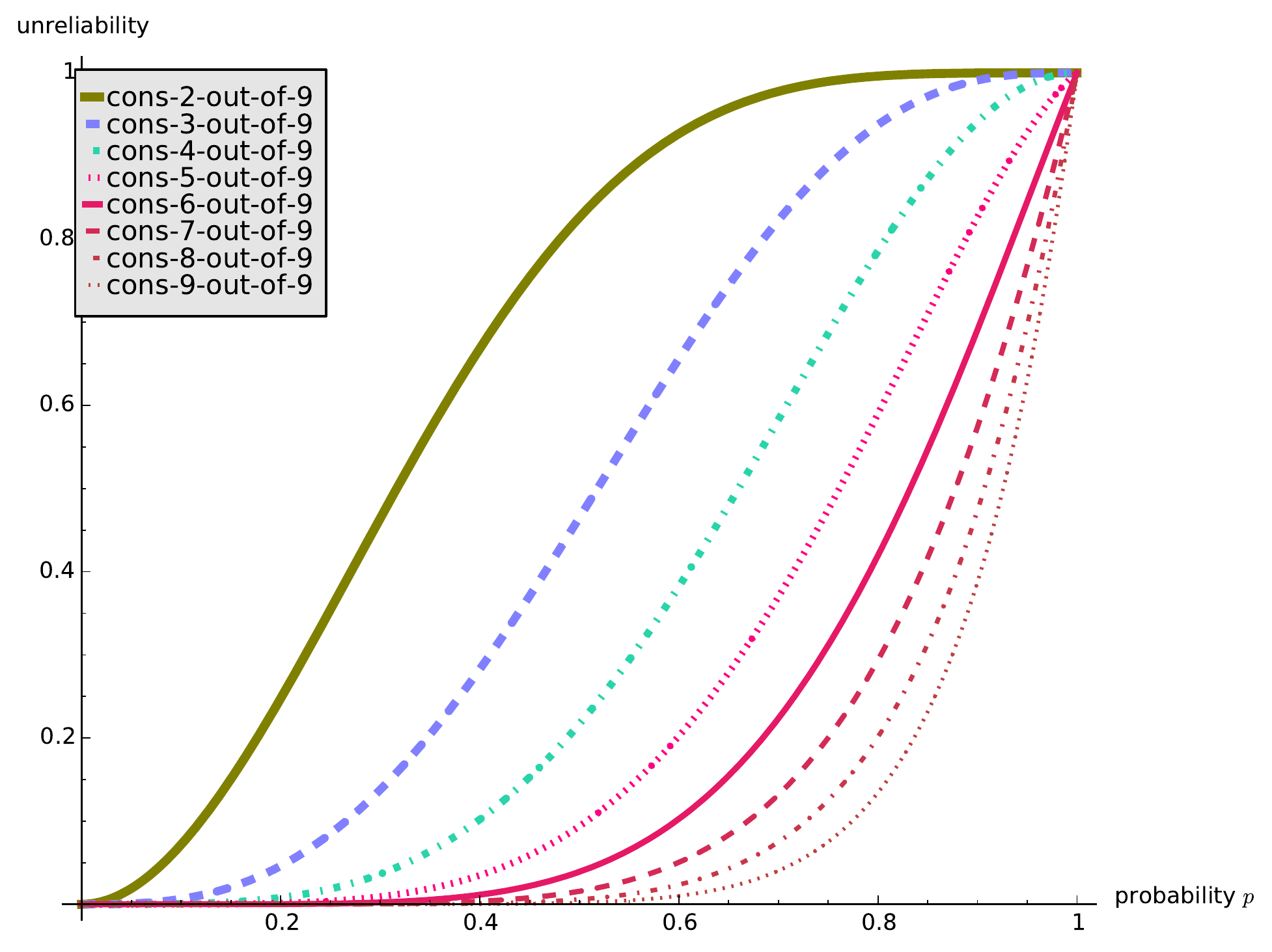}
\includegraphics[scale=0.35]{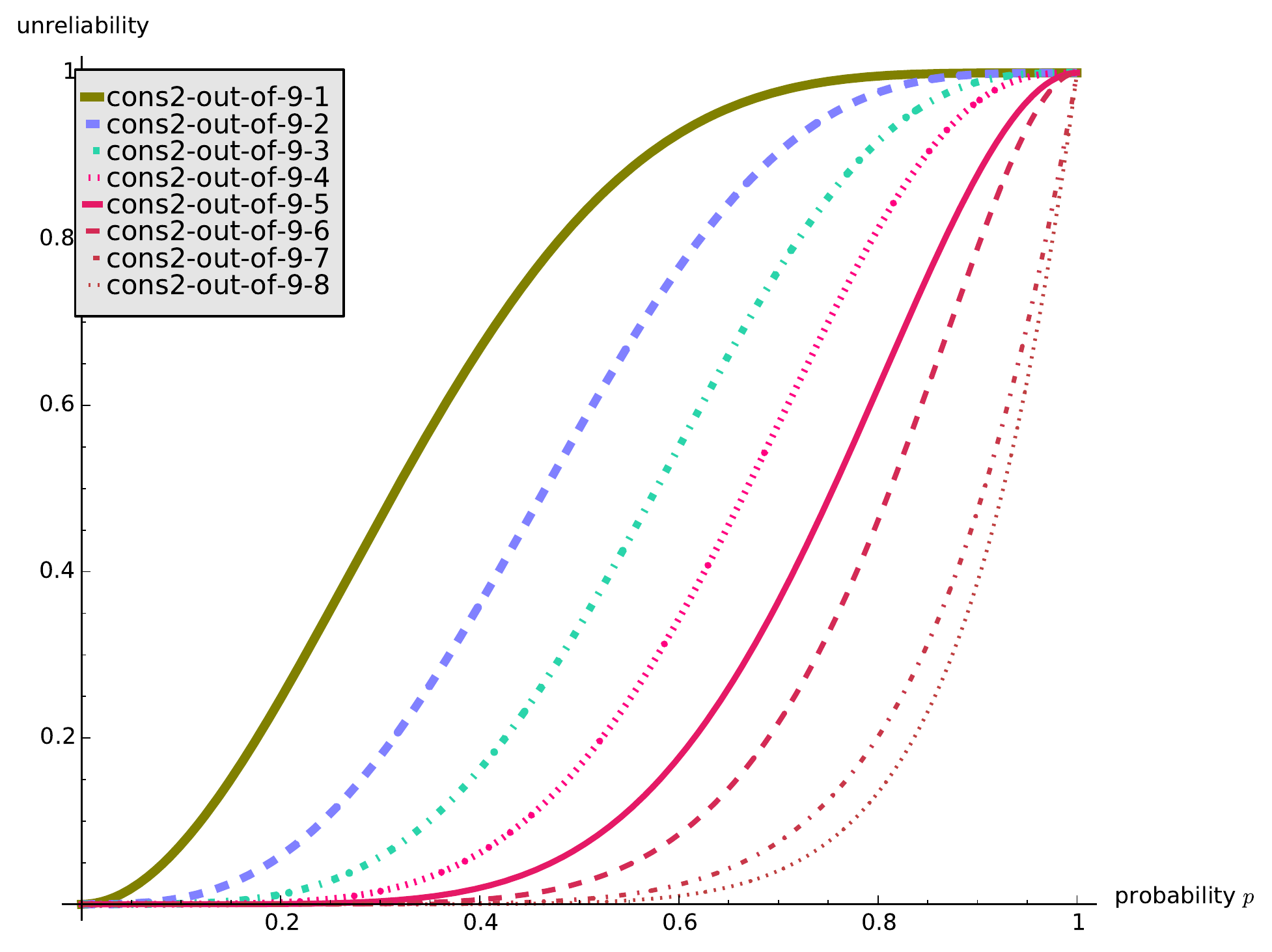}
\caption{Unreliability of Cons/$i$-out-of-$9$:F systems for $2\leq i\leq 9$ and probability of $i$ simultaneous failure events of Cons/$2$-out-of-$9$:F systems for $1\leq i\leq 8$.}
\label{fig:consUnrel}
\end{center}
\end{figure}

If we consider the probability of $i$ simultaneous failure events for a Cons/$k$-out-of-$n$:F system with i.i.d components we observe that this system does not show the staircase behaviour of ordinary $k$-out-of-$n$:F systems. This is seen in Figure \ref{fig:unreliabilityfiltrationcons}, that shows how the probability that multiple failures appear in the i.i.d. Cons/$2$-out-of-$9$:F system for three different component probabilities $p=0.2,\,0.5$ and $0.8$. 

\begin{figure}[t]
\begin{center}
\includegraphics[scale=0.35]{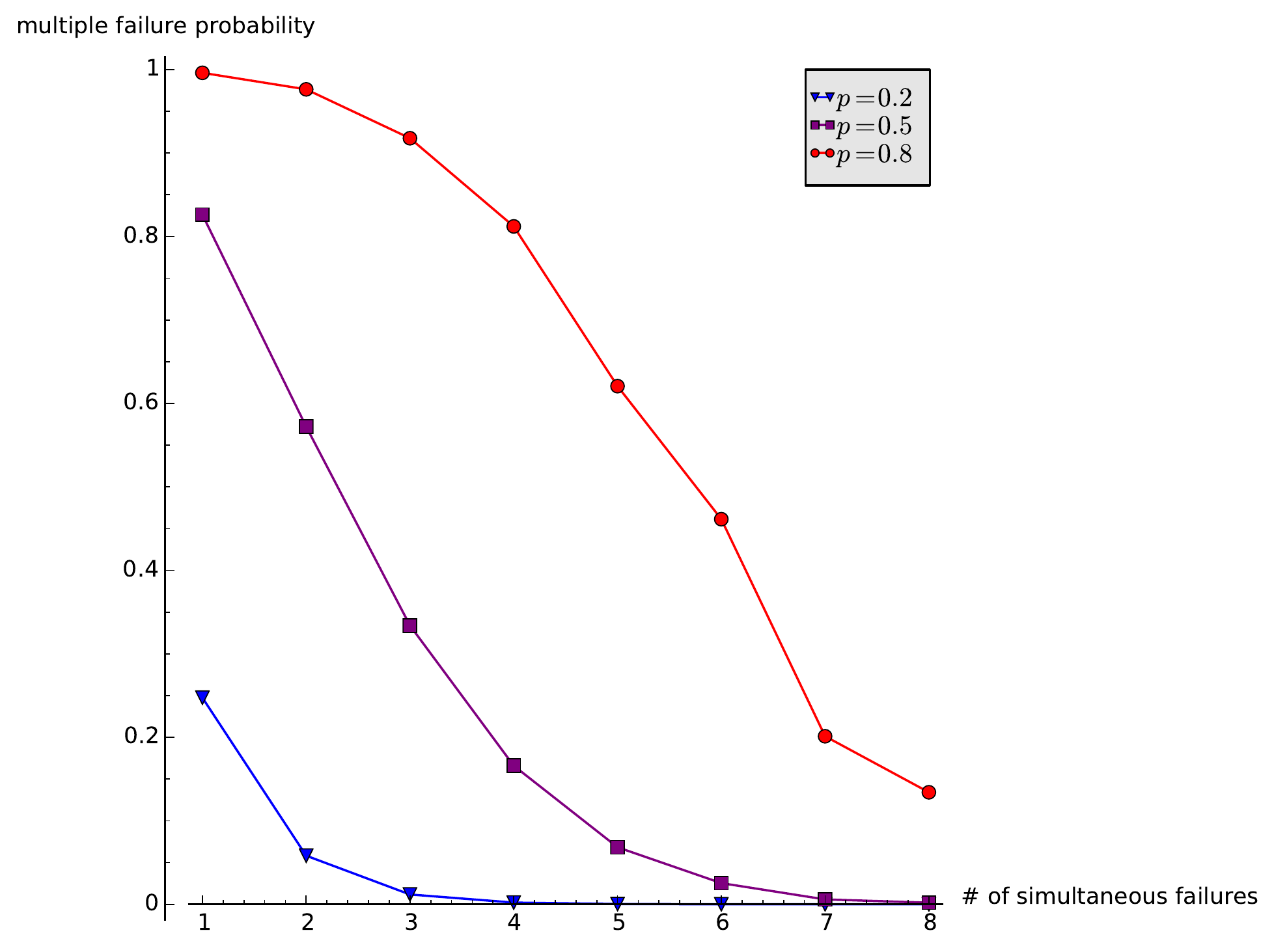}
\caption{Probability of simultaneous failures for i.i.d. Cons/$2$-out-of-$9$:F systems with different component failure probabilities.}\label{fig:unreliabilityfiltrationcons}
\end{center}
\end{figure}
}
\end{Example}

\section {Computational experiments} \label{sec:computations}


\subsection {Number of generators} \label{sub:numgens}
One of the main drawbacks of the use of enumerative methods for the reliability analysis of systems with multiple simultaneous failures is that the number of possible combinations of such simultaneous failures grows exponentially. For a system with $r$ minimal failure events (cuts) the number of $i$-multicuts can be ${r \choose i}$. However, in most systems, the actual number is much lower, there is redundancy in those ${r\choose i}$ combinations of cuts. We then face two problems: the big number of possible cut combinations and how to discard the redundant ones. 

In the case of $k$-out-of-$n$:F and Cons/$k$-out-of-$n$:F systems, our methods efficiently enumerate the minimal $i$-multicuts of the system. Since $k$-out-of-$n$:F systems with $i$ simultaneous failures correspond to another $j$-out-of-$n$:F system, we can count the generators of the corresponding ideal using Theorem~\ref{thm:kn}. The corresponding number of generators is much smaller than the possible number ${ {n\choose k} \choose i}$ as indicated in Table \ref{table:genskn} for the $2$-out-of-$8$ system (showing the systems up to 10 simultaneous failures).

\begin{table*}[!t]
\footnotesize
\begin{center}
\begin{tabular}{c|c|c|c|c|c|c|c|c|c|c}
$i$&1&2&3&4&5&6&7&8&9&10\\
\hline
${ {n\choose k} \choose i}$&28&378&3276&20475&98280&376740&1184040&3108105&6906900&13123110 \\
\hline
${n \choose j}$&28&56&56&70&70&70&56&56&56&56\\
\end{tabular}\caption{Number of minimal generators for the ideals of $2$-out-of-$8$ system with $i$ simultaneous failures, $i=1\dots,10$.}
\label{table:genskn}
\end{center}
\end{table*}

The list of generators of the iterated Cons/$k$-out-of-$n$:F ideals are given by Theorem~\ref{thm:conskn}. Counting them for given $k$, $n$ and $i$ amounts to count all the subsets of $\{1,\dots, n-k+1\}$ of cardinality $i$ having no gaps of sizes between $1$ and $k-1$.
\begin{itemize}
\item[-]Since the sets we are considering are sorted, the initial element can only be $1$ up to $n-k-i+1$. For each such initial element $m$, we can consider the sum of all the gaps in the set to be at most $n-k-i+1$.
\item[-] Each of such total gaps can be obtained using at most $i-1$ summands (the gaps in a set of cardinality $i$ can at most be $i-1$) each of which has to be at least $k$. For each number $a\leq i-1$ of summands this is equivalent to count all the compositions of the integer $k$ in at most $i$ parts of smallest size $k$. This number is known to be ${{m-(k-1)a-1} \choose {a-1}}$, cf. \cite{A84,KM13}.
\item[-] Each of these compositions can be used in ${{i-1}\choose{i-1-a}}$ ways.
\end{itemize}
From these considerations we obtain the following formula for the number of generators of the $i$-th $\lcm$-ideal of a Cons/$k$-out-of-$n$:F system. This formula has a summand corresponding to the sets in $\Sc^i_0$ and another one to those in $\Sc^i_m$ for $m>0$.
\begin{Proposition}\label{prop:numgensconskni}
The number of generators of the $i$-fold $\lcm$-ideal of a Cons/$k$-out-of-$n$:F system is
\begin{eqnarray*}
N_{k,n,i}&=\sum_{m=1}^{n-k-i+1}(n-k-i+2-m)
\sum_{a=1}^{i-1}{{i-1}\choose{i-1-a}}{{m-(k-1)a-1}\choose{a-1}}\\
&+(n-k-i+2).
\end{eqnarray*}
\end{Proposition}

Table \ref{table:gensconskn} shows that the numbers of generators in this case are also much smaller than the corresponding binomial coefficients. The time needed for enumerating these sets (i.e. the minimal $i$-multicuts for all $i$) is not neglectable when compared to the time needed to compute the unreliability polynomial and bounds, as explained in \ref{subsec:comprel} . This makes the problem of counting the minimal multicuts more relevant.

\begin{table*}[!t]
\footnotesize
\begin{center}
\begin{tabular}{c|c|c|c|c|c|c|c|c|c|c}
$i$&1&2&3&4&5&6&7&8&9&10\\
\hline
${ {n-k+1} \choose i}$&16&120&560&1820&4368&8008&11440&12870&11440&8008 \\
\hline
$\#generators$&16&70&124&151&148&126&100&79&56&34\\
\end{tabular}\caption{Number of minimal generators for Cons/$5$-out-of-$20$:F ideals with $i$ simultaneous failures, $i=1\dots,10$.}
\label{table:gensconskn}
\end{center}
\end{table*}


\subsection{Computing reliabilities}\label{subsec:comprel}
We now focus on computing the unreliability and bounds for systems with multiple simultaneous failures. For structured systems such as $k$-out-of-$n$ or series-parallel systems, there are several well-known formulae and methods (cf. \cite{KZ03} and the survey on Cons/$k$-out-of-$n$ systems in \cite{DB15} for the i.i.d. case). When the systems have no apparent structure or no known formulae for reliability and bounds, the authors proposed the algebraic method in \cite{SW09} which can be applied in general and is based on computing the Hilbert series of a monomial ideal. Although this is a \#P-hard problem (in fact it is in the same complexity class of the computation of the reliability of a system, cf. \cite{B86}), there are good algorithms available in computer algebra systems.
We present in Table \ref{table:computations_2-20} computing times for simultaneous failures of the Cons/$2$-out-of-$20$:F system. The first column shows the number $i$ of simultaneous failures for this system. The second column shows the number of minimal generators of the corresponding monomial ideal $J_{k,n,i}$ i.e. the number of minimal $i$-multicuts (obtained using Proposition \ref{prop:numgensconskni}). Columns {\tt Binomials} and {\tt Formula} show the times used for computing the set of minimal generators of the ideal with two different methods. The first one uses all ${19 \choose i}$ possible combinations of $i$ of the $19$ generators of $J_{2,20}$ and then eliminates the redundant ones. The second one makes use of Theorem \ref{thm:conskn}. The difference between these two columns is particularly significant at the central rows of the table, where the number of generators is bigger and so is the corresponding binomial number ${19 \choose i}$. Theorem \ref{thm:conskn} was implemented in the computer algebra system Macaulay2 (\cite{M2}).
The last three columns of  the table show the timings of the reliability-related computations. Column {\tt Hilbert} shows the computation of the Hilbert series using the algorithms in Macaulay2. The Hilbert series provides the reliability polynomial but in a form that is not suitable for obtaining bounds by truncating terms of the polynomial. Column {\tt Resolution} shows times for computing the minimal free resolution of the corresponding ideals using Macaulay2. This gives a form of the reliability polynomial suitable for bounds. Also suitable for bounds is the form of the reliability polynomial given by the Mayer-Vietoris resolution (see \cite{S09}). The time of computation of this resolution (which is often minimal) is given in column {\tt MVT} using the algorithm in CoCoALib (\cite{CoCoALib}). All times are in seconds and all computations were perfumed on an HP Zbook with Intel Core i7 processor and 16Gb RAM. Observe that the fastest computation is that of the Hilbert function, which provides less information than the other two. Among these two complete computations (which provide reliability formulae and bounds) {\tt MVT} is much faster, in particular when few simultaneous failures occur. The reason is that the {\tt MVT} algorithm is more sensitive to the growth in number of generators than the Macaulay2 algorithm for minimal resolutions which is more sensitive to longer resolutions. Since the length of the resolution is approximately $n-i$, the first part of the table has longer resolutions, i.e., the times in {\tt Resolution} column grow higher, and in the central part the number of generators is much bigger and the difference between {\tt Resolution} and {\tt MVT} is reduced. For the usual case of few simultaneous failures, the Mayer-Vietoris algorithm is more convenient.

\begin{table*}[!t]
\footnotesize
\begin{center}
\begin{tabular}{c|r|r|r|r|r|r}
$i$&\# of gens. & Binomials & Formula &{\tt Hilbert}& {\tt Resolution}&{\tt MVT}\\
\hline
1&19&0.001&0.001&0.209&2.139&0.024\\
2&154&0.008&0.024&0.096&27.840&0.660\\
3&712&0.063&0.103&0.168&81.742&1.432\\
4&2138&0.324&0.651&0.333&144.514&2.676\\
5&4537&1.753&1.163&0.524&201.975&6.452\\
6&7248&4.531&2.242&0.753&195.633&14.424\\
7&9143&10.012&5.086&0.746&148.763&25.952\\
8&9434&18.112&5.208&0.725&103.777&24.664\\
9&8169&32.283&7.065&0.556&60.028&15.932\\
10&6046&34.076&4.138&0.336&23.037&7.036\\
11&3874&33.413&2.566&0.246&10.849&2.652\\
12&2164&25.540&2.025&0.075&4.573&0.780\\
13&1067&15.375&1.501&0.039&0.484&0.176\\
14&448&6.946&0.263&0.016&0.097&0.032\\
15&180&2.086&0.113&0.006&0.014&0.012\\
16&49&0.449&0.032&0.002&0.002&0\\
17&19&0.082&0.011&0.001&0.001&0\\
18&2&0.001&0.001&0&0&0\\
19&1&0.001&0.001&0&0&0\\
\hline
{\bf Total:}& -- &185.065&32.194&4.831&1005.168&102.908\\

\end{tabular}\caption{Time for computing the generators, Hilbert series and resolutions of multiple failure ideals of the Cons/$2$-out-of-$20$ system.}
\label{table:computations_2-20}
\end{center}
\end{table*}

\section{Conclusions}
The main computational tasks that one has to face when computing the reliability of systems subject to multiple simultaneous failures are two: the first is to compute the minimal failure events, which grow exponentially with respect to the number of minimal failures of the initial system. In algebraic terms, this is equivalent to the enumeration of the minimal generating set of the corresponding monomial ideals. The second task is to obtain reliability formulae and bounds. 

In systems with low redundancy, such as consecutive $2$-out-of-$n$ for large $n$, the first problem is computationally expensive. In fact it can be even more expensive than the actual reliability formulae computations. This can be seen by comparing the {\tt Total} row of columns {\tt Binomial} and {\tt MVT} in Table \ref{table:computations_2-20}. The importance of results such as Theorem \ref{thm:conskn} and Theorem \ref{thm:kn} is clearer at the sight of the computation times.

With respect to reliability computations, a general approach such as the algebraic one can be useful when closed form formulae are not available, in particular when the components have different failure probabilities, which is the main assumption in many approaches, cf. \cite{DB15}. Since this method is general, improvements in its implementations, like the {\tt MVT} algorithm in CoCoALib may result in a useful tool for reliability analysis.
\bibliographystyle{betta}

\bibliography{refs}

\end{document}